\documentclass{article}
\usepackage[utf8]{inputenc}

\title{An upper bound for the solving degree in terms of the degree of regularity}
\author{Flavio Salizzoni}
\date{}

\usepackage{geometry}
\usepackage{amssymb,amsmath,amsthm}
\usepackage{graphicx,url,comment}
\usepackage{mathrsfs}

\theoremstyle{definition}
\newtheorem{theorem}{Theorem}[section]
\newtheorem{proposition}[theorem]{Proposition}
\newtheorem{lemma}[theorem]{Lemma}
\newtheorem{definition}[theorem]{Definition}
\newtheorem{example}[theorem]{Example}
\newtheorem{remark}[theorem]{Remark}
\newtheorem{corollary}[theorem]{Corollary}

\newcommand{\N}{\mathbb N}

\newcommand{\F}{\mathbb F}

\newcommand{\Ff}{\mathcal F}

\newcommand{\sd}{\mathrm{sd}_{\sigma}}
\newcommand{\dreg}{\mathrm{d}_{\mathrm{reg}}}
\newcommand{\Vd}{V_{\mathcal{F},d}}
\newcommand{\Vdd}{V_{\mathcal{F},d+1}}

\newcommand{\lfd}{\mathrm{Lfd}}

\begin{document}

\maketitle
\begin{abstract}
    The solving degree is an important parameter for estimating the complexity of solving a system of polynomial equations.
    In this paper, we provide an upper bound for the solving degree in terms of the degree of regularity. We also show that this bound is optimal. As a direct consequence, we prove an upper bound for the last fall degree and a Macaulay bound.
\end{abstract}
\section{Introduction}
Many problems can be modeled through a system of polynomial equations. Solutions to such a system can be found in polynomial time once we know a reduced Gr\"obner basis of the system. Nowadays, there are many algorithms for the computation of Gr\"obner bases, among which the most efficient ones belong to the family of linear-algebra based algorithms. This family includes among others F4, F5, and XL algorithms \cite{Courtois,F4,F5}. Estimating their computational complexity has become increasingly important because of their applications in post-quantum cryptography.

Given a linear-algebra based algorithm $H$ for solving a system of polynomial equations $\Ff$, the solving degree for $\Ff$ with respect to $H$ is defined as the largest degree of a polynomial that appears while we solve $\Ff$ using $H$. The interest in this parameter stems from the fact that an upper bound on it, results in an upper bound on the complexity of these algorithms. However, computing explicitly the solving degree is in general extremely difficult. For this reason, it is useful to provide upper bounds for the solving degree in terms of other parameters, such as the degree of regularity, the last fall degree, and the Castelnuovo-Mumford regularity. For instance, in \cite{MHKY18} Huang, Kosters, Yang, and Yeo described an algorithm for solving polynomial systems whose complexity is bounded from above by the last fall degree. In \cite{CG}, Caminata and Gorla proved that the solving degree for a family $\Ff=\{f_1,\dots,f_k\}$ is upper bounded by the Castelnuovo-Mumford regularity of the family $\Ff^h=\{f_1^h,\dots,f_k^h\}$, where $f_i^h$ is the homogenization of $f_i$.

The degree of regularity $\dreg(\Ff)$ was introduced by Bardet, Faugère, and Salvy in \cite{Bar04,BFS}, where they also provide an upper bound for this degree for cryptographic semi-regular sequences. Even if many authors use the degree of regularity as a heuristic upper bound for the solving degree, it is not completely clear what is the exact relation between these two parameters. For instance, see \cite[Examples 4.7]{CG} and \cite[Examples 2.2, 2.3 and 2.4]{BNGMT}. In \cite{ST,Tenti}, Semaev and Tenti proved that under certain conditions the solving degree is bounded from above by $2\dreg(\Ff)-2$ (see \cite[Corollary 3.67]{Tenti} and the discussion below \cite[Theorem 2.1]{ST}).

In this paper we consider the solving degree $\sd$ for a family of algorithms that contains among others MutantXL \cite{Buchmann09} and MXL2 \cite{Buchmann08}. The main result is the following bound for the solving degree in terms of the degree of regularity.
\begin{theorem}\label{theorem:boundsolvdeg}
Let $\mathcal{F}=\{f_1,\dots,f_n\}$ be a family of polynomials such that $\max\deg(f_i)\leq \dreg(\mathcal{F})$ and let $\sigma$ be a degree-compatible term order. Then, $\sd(\Ff)\leq \dreg(\Ff)+1$.    
\end{theorem}
We prove that this bound is optimal showing that for every $1<d\in\N$ there exists a family $\Ff$ such that $\dreg(\Ff)=d$ and $\sd(\Ff)=d+1$. As a consequence of Theorem \ref{theorem:boundsolvdeg}, we obtain that the last fall degree is always bounded from above by the degree of regularity. 
Finally, we improve the Macaulay bound for a family of polynomials with finite degree of regularity.
\section{Preliminaries}
Let $\mathbb{K}$ be a field (not necessarily finite). For an ideal $\mathcal{I}$ in $\mathbb{K}[x_1,\dots,x_n]$, we denote by $\mathcal{I}_d$ the set of all homogeneous polynomial of degree $d$ in $\mathcal{I}$. Let $\Ff=\{f_1,\dots,f_k\}$ be a finite family of $\mathbb{K}[x_1,\dots,x_n]$ and let $(\Ff)$ be the ideal generated by $\Ff$.
\begin{definition}
    Fix a term order $\sigma$. We denote by $\mathrm{Gbd}_{\sigma}(\Ff)$ the maximum degree of an element of a reduced Gr\"obner basis of $(\Ff)$ with respect to $\sigma$.
\end{definition}

We now recall the definition of degree of regularity, as was originally given in \cite{Bar04,BFS} by Bardet, Faugère, and Salvy. For a polynomial $p\in\mathbb{K}[x_1,\dots,x_n]$, let $p^{\mathrm{top}}$ be the homogeneous part of largest degree of $p$.
\begin{definition}
 The degree of regularity $\dreg(\Ff)$ of $\Ff$ is the smallest integer $d$ for which $$\left(\Ff^{\mathrm{top}}\right)_d=(\mathbb{K}[x_1,\dots,x_n])_d,$$
 where $\Ff^{\mathrm{top}}=\left\{f_1^{\mathrm{top}},\dots,f_k^{\mathrm{top}}\right\}$. If such a $d$ does not exist, we set $\dreg(\Ff)=+\infty$.
\end{definition}
\begin{remark}\label{remark:grobnerdreg}
If the term order $\sigma$ is degree-compatible, then $\mathrm{Gbd}_{\sigma}(\Ff)\leq \dreg(\Ff)$, as it has been already observed in \cite[Remark 4.6]{CG} by Caminata and Gorla. 
\end{remark}
Another important parameter for a family $\Ff$ is the last fall degree, which was defined for the first time in \cite{MHKY15} by Huang, Kosters, and Yeo.
\begin{definition}
    For an integer $d$, we denote by $V_{\Ff,d}$ the smallest $\mathbb{K}$-linear space with respect to inclusion such that
    \begin{itemize}
        \item $\{f\in\Ff:\deg(f)\leq d\}\subseteq V_{\Ff,d}$,
        \item for every $m$ monomial in $\mathbb{K}[x_1,\dots,x_n]$ and every $f\in V_{\Ff,d}$ such that $\deg(mf)\leq d$, then $mf\in V_{\Ff,d}$.
    \end{itemize}
    We set $V_{\Ff,+\infty}=(\Ff)$. The last fall degree $\lfd(\Ff)$ is the minimal $d\in\mathbb{Z}_{>0}\cup \{+\infty\}$ such that $f\in V_{\Ff,\max\{d,\deg(f)\}}$ for all $f\in(\Ff)$.
\end{definition}
The solving degree is usually defined starting from the Macaulay matrix $M(\Ff,d)$ to which we add a certain number of rows, see for instance \cite[Section 2.2]{GMP22} for an explicit discussion. However, for the purpose of this paper, it is more convenient for us to define the solving degree directly in terms of the spaces $V_{\Ff,d}$. Notice that the following definition is equivalent to the standard one, as shown in \cite[Theorem 1]{GMP22}.   
\begin{definition}
    Let $\sigma$ be a degree-compatible term order. The solving degree $\sd(\Ff)$ of $\Ff$ is the smallest $d$ for which $V_{\Ff,d}$ contains a Gr\"obner basis of $\Ff$ with respect to $\sigma$.
\end{definition}
In \cite{CG22}, Caminata and Gorla proved the following result that highlights a relation between the solving degree and the last fall degree.
\begin{theorem}\label{theore:lfdgorla}
    Let $\Ff$ be a polynomial system and let $\sigma$ be a degree-compatible term order. Then
    $$\sd(\Ff)=\max\{\lfd(\Ff),\mathrm{Gbd}_{\sigma}(\Ff)\}.$$
\end{theorem}
 For a more detailed discussion about the solving degree, the degree of regularity, and the last fall degree we refer the interested reader to \cite{CG,CG22}.

\section{Proof of Theorem \ref{theorem:boundsolvdeg}}
From now on, let $\Ff=\{f_1,\dots,f_k\}$ be a family of polynomials in $\mathbb{K}[x_1,\dots,x_n]$ such that $\dreg(\Ff)=d<+\infty$ and $\deg(f_i)\leq d$ for all $1\leq i\leq k$.
\begin{lemma}\label{lemma:degreg}
 For each monomial $m$ of degree $d$ there exists $p\in\Vd$ such that $p^{\mathrm{top}}=m$.
\end{lemma}
\begin{proof}
Since $d$ is the degree of regularity of $\Ff$, there exist $m_1,\dots,m_{k}$ monomials such that $\sum_{i=1}^km_if_i^{\mathrm{top}}=m$, $m_i\neq0$, and $\deg(m_if_i)\leq d$ for all $1\leq i\leq k$. Let $p=\sum_{i=1}^km_if_i$. It is clear that $\deg(p)=\deg(m)=d$, $p\in \Vd$, and $p^{\mathrm{top}}=m$.
\end{proof}
Let $D=\binom{d+n-1}{d}$ and let $\mathcal{P}_{\Ff}=\{p_1,\dots,p_D\}\subseteq \Vd$ be a set such that for each monic monomial $m$ of degree $d$ there exists $1\leq i\leq D$ such that $p_i^{\mathrm{top}}=m$. The existence of $\mathcal{P}_{\Ff}$ follows from Lemma \ref{lemma:degreg}. Notice that $\mathcal{P}_{\Ff}\subseteq \Vdd$, since $\Vd\subseteq\Vdd$.
\begin{lemma}\label{lemma:writingpol}
For every $f\in\Vdd$ of degree $d$, there exist $a_1,\dots,a_D\in\mathbb{K}$ and $\bar f\in\Vdd$ with $\deg(\bar f)<d$, such that $f=\bar f+\sum_{i=1}^D a_ip_i$, where $p_i\in \mathcal{P}_{\Ff}$.
\end{lemma}
\begin{proof}
    Since $\langle p_1^{\mathrm{top}},\dots,p_D^{\mathrm{top}}\rangle_{\mathbb{K}}=\mathbb{K}[x_1,\dots,x_n]_d$, there exist $a_1,\dots,a_d\in\mathbb{K}$ such that $f-\sum_{i=1}^D a_ip_i=\bar f$ has degree strictly smaller than $d$. As $f,p_1,\dots,p_D\in\Vdd$, we conclude that $\bar f\in\Vdd$ too.
\end{proof}
We denote by $\mathscr{P}_{\mathrm{fin}}(\mathbb{K}[x_1,\dots,x_n]^3)$ the set of all finite subsets of $\mathbb{K}[x_1,\dots,x_n]^3$. We now define a subset $A_d$ of $\mathscr{P}_{\mathrm{fin}}(\mathbb{K}[x_1,\dots,x_n]^3)$ giving a concrete description of its elements. A set $\{(m_1,a_1,g_1),\dots, (m_{\ell},a_{\ell},g_{\ell})\}\in A_d$ if and only if the followings are satisfied:
\begin{itemize}
    \item there is exactly one element such that $m_i=1$. For simplicity, we assume $m_1=1$. Then, $a_1=1$ and $g_1\in \Vdd$,
    \item $m_2,\dots, m_{\ell}$ are monic monomials of degree at least $2$,
    \item $a_2,\dots,a_{\ell}\in\mathbb{K}\setminus0$ and $g_2,\dots,g_{\ell}\in\mathcal{P}_{\Ff}$,
    \item if $m_i=m_j$ then $g_i\neq g_j$.
\end{itemize}
For $X,Y\in A_d$, we define $X\oplus Y\in A_d$ as follows. If $(1,1,f_x)\in X$ and $(1,1,f_y)\in Y$, then $(1,1,f_x+f_y)\in X\oplus Y$. Moreover, for $m$ of degree at least $2$, we have that $(m,a,p)\in X\oplus Y$ if and only if and only if one of the following occurs:
\begin{itemize}
    \item $(m,a,p)\in X$, and for all $b\in \mathbb{K}\setminus 0$ we have that$(m,b,p)\notin Y$,
    \item $(m,a,p)\in Y$, and for all $b\in \mathbb{K}\setminus 0$ we have that $(m,b,p)\notin X$,
    \item there exist $(m,b,p)\in X$ and $(m,c,p)\in Y$ such that $a=b+c\neq0$.
\end{itemize}
\begin{lemma}\label{lemma:group}
    We have that $(A_d,\oplus)$ is an abelian group whose identity element is $\{(1,1,0)\}$.
\end{lemma}
\begin{proof}
    The fact that $\oplus$ is associative and commutative comes from the fact that $+$ is associative and commutative over $\mathbb{K}$. Moreover, it is clear from the definition that $\{(1,1,0)\}\oplus X=X$ for all $X\in A_d$. Finally, let $X=\{(1,1,f_x),(m_2,a_2,g_2),\dots,(m_{\ell},a_{\ell},g_{\ell})\}\in A_d$, then the set $\{(1,1,-f_x),(m_2,-a_2,g_2),\dots,(m_{\ell},-a_{\ell},g_{\ell})\}$ is also an element of $A_d$. We denote this set by $-X$. It is easy to verify that $X\oplus -X=\{(1,1,0)\}$.
\end{proof}
For $X=\{(1,1,f_x),(m_2,a_2,g_2),\dots (m_{\ell},a_{\ell},g_{\ell})\}\in A_d$ and $a\in \mathbb{K}$, we define $a\star X$ as
\begin{equation*}
    a\star X=\begin{cases}
    \{(1,1,af_x),(m_2,aa_2,g_2),\dots (m_{\ell},aa_{\ell},g_{\ell})\}&a\neq 0\\
    \{(1,1,0)\}&a=0
\end{cases}
\end{equation*}
Clearly, we have that $a\star X\in A_d$. One can also verify that $(A_d,\oplus,\star)$ is a vector space over $\mathbb{K}$.

For $X\in A_d$ we define the degree of $X$ as $\deg(X)=\max\deg(m_ig_i)$. Directly from the definition of $\oplus$ and of $\star$, we obtain the following lemma.
\begin{lemma}\label{lemma:sumdegree}
    Let $X,Y\in A_d$ and $0\neq a\in \mathbb{K}$, then $\deg(X\oplus Y)\leq \max\{\deg(X),\deg(Y)\}$ and $\deg(a\star X)=\deg(X)$.
\end{lemma}
For $f\in\mathbb{K}[x_1,\dots,x_n]$, let $A_{d,f}$ be the subset of $A_d$ given by
\begin{equation*}
    A_{d,f}=\left\{\{(m_1,a_1,g_1),\dots, (m_{\ell},a_{\ell},g_{\ell})\}\in A_d: \sum_{i=1}^{\ell} m_ia_ig_i=f\right\}.
\end{equation*}
Notice that the sets $A_{d,f}$ with $f\in\Ff$ form a partition of $A_d=\bigsqcup_f A_{d,f}$. The following lemma shows how the operations $\oplus$ and $\star$ act on this partition. We omit the proof since it is straightforward.
\begin{lemma}\label{lemma:sumofsets}
    Let $X\in A_{d,f_x}$, $Y\in A_{d,f_y}$, and $0\neq a\in\mathbb{K}$. Then $X\oplus Y\in A_{d,f_{x}+f_{y}}$ and $a\star X\in A_{d,af_x}$.
\end{lemma}
Our goal now is to prove that for every $f\in(\Ff)$, the corresponding $A_{d,f}$ is not empty. To achieve this result, we need one last lemma.
\begin{lemma}\label{lemma:multiple}
Let $m$ be a monic monomial. If $f\in \Vdd$ and $\deg(f)\leq d$, then there exists $X\in A_{d,mf}$ such that $\deg(X)=\deg(mf)$.
\end{lemma}
\begin{proof}
    We proceed by induction on $\deg(mf)$. If $\deg(mf)\leq d+1$, since $f\in\Vdd$ we have $mf\in \Vdd$ and therefore $X=\{(1,1,mf)\}\in A_{d,mf}$. So, suppose $\deg(mf)>d+1$. If $\deg(f)<d$, there exist $m_1,m_2$ monic monomials such that $m=m_1m_2$ and $\deg(m_1f)=d$. Since $m_1f\in\Vdd$, without loss of generality, we can assume $\deg(f)=d$.
    
    By Lemma \ref{lemma:writingpol}, we write $f=\bar f+\sum_{i=1}^D a_ip_i$ with $\deg(\bar f)<d$. By inductive hypothesis we have that there exists $\bar X\in A_{d,m\bar f}$ such that $\deg(\bar X)\leq\deg(m\bar f)$. Let $f'=\sum_{i=1}^D a_ip_i$. It is clear that $X'=\{(1,1,0)\}\cup\{(m,a_i,p_i):a_i\neq 0\}\in A_{d,mf'}$ and that $\deg(X')= \deg(mf')$. Since $mf=m\bar f+mf'$, by Lemma \ref{lemma:sumofsets} we conclude that $\bar X\oplus X'\in A_{d,mf}$ and that $\deg(X\oplus X')=\max\{\deg(m\bar f),\deg(mf')\}= \deg(mf)$.
\end{proof}
\begin{proposition}\label{proposition:notemptyfamily}
Let $f\in (\Ff)$, then $A_{d,f}\neq\emptyset$.
\end{proposition}
\begin{proof}
    Since $f\in(\Ff)$, there are $q_1,\dots,q_k\in \mathbb{K}[x_1,\dots,x_n]$ such that $f=\sum q_if_i$. Each $q_if_i$ can be written as $\sum m_{i,j}a_{i,j}f_i$ with $m_{i,j}$ monic monomial. Since for each $i,j$, $\deg(a_{i,j}f_i)\leq d$, by Lemma \ref{lemma:multiple} there exists $X_{i,j}\in A_{d,a_{i,j}m_{i,j}f_i}$. By Lemma \ref{lemma:sumofsets} we conclude that $\sum_{i,j}X_{i,j}\in A_{d,f}$. 
\end{proof}
The next theorem is central in the proof of Theorem \ref{theorem:boundsolvdeg}.
\begin{theorem}\label{theorem:finVdd}
If $f\in (\Ff)$ and $\deg(f)\leq d+1$, then $f\in\Vdd$.
\end{theorem}
\begin{proof}
    Let $\delta=\min_{X\in A_{d,f}}\deg X$ and $A_{\delta}=\{X\in A_{d,f}:\deg(X)=\delta\}$. By Proposition \ref{proposition:notemptyfamily} we have that $A_{d,f}\neq\emptyset$ and so $A_{\delta}\neq\emptyset$. For $X,Y\in A_{\delta}$ we say that $X\leq Y$ if the greatest monic monomial appearing in $X$ is smaller than the greatest monic monomial appearing in $Y$ (in  degree lexicographic order with: $x_1>x_2>\dots>x_n$). Let $X=\{(m_1,a_1,g_1),\dots, (m_{\ell},a_{\ell},g_{\ell})\}$ be a minimal element of $A_{\delta}$ according to this order.
    If $\delta\leq d+1$, then $X=\{(1,1,f)\}$ and this implies $f\in \Vdd$.
    
    If $\delta\geq d+2$, suppose for instance that $m_2=x_1^{b_1}\dots x_n^{b_n}$ is maximal among the monomials of $X$ and $g_2=x_1^{c_1}\dots x_n^{c_n}+\bar g_2$ where $\deg(\bar g_2)<d$. Since $\deg(f)\leq d+1$, there exists an index $i$ such that $(m_ig_i)^{\mathrm{top}}=(m_2g_2)^{\mathrm{top}}$ and $m_i< m_2$. This implies that
    $$\bar i=\min\{i:b_i\neq0\}<\max\{j:c_j\neq0\}=\bar j.$$ Let $p\in\mathcal{P}_{\Ff}$ such that $p^{\mathrm{top}}=x_1^{c_1}\dots x_{\bar i}^{c_{\bar i}+1}\dots x_{\bar j}^{c_{\bar j}-1}$ and $m=x_{\bar i}^{b_{\bar i}-1}\dots x_{\bar j}^{b_{\bar j}+1}\dots x_n^{b_n}$. Then, we can write $f$ as 
    $$f=\sum a_im_ig_i +a_2mp-a_2mp=x_{\bar i}^{b_{\bar i}-1}\dots x_{\bar j}^{b_{\bar j}}\dots x_n^{b_n}a_2(x_{\bar i}g_2-x_{\bar j}p)+\sum_{i\neq 2} a_im_ig_i+a_2mp.$$
    Notice that $\deg(x_{\bar i}g_2-x_{\bar j}p)\leq d$, and $g=x_{\bar i}g_2-x_{\bar j}p\in\Vdd$ since $g_2,p\in\Vd$. So, by Lemma~\ref{lemma:multiple}, there exists $X_1\in A_{d,a_2x_{\bar i}^{b_{\bar i}-1}\dots x_{\bar j}^{b_{\bar j}}\dots x_n^{b_n}g}$ such that $\deg(X_1)\leq \delta-1$. Let $h=\sum_{i\neq 2} a_im_ig_i$, then 
    $X_2=\{(1,1,0),(m_3,a_3,g_3),\dots,(m_\ell,a_{\ell},g_{\ell})\}\in A_{d, h}$ with $\deg(X_2)\leq\delta$. Obviously $X_3=\{(1,1,0),(m,a_2,p)\}\in A_{d,a_2mp}$. By Lemma \ref{lemma:sumofsets}, $X_1\oplus X_2\oplus X_3\in A_{d,f}$.  Since $m$ is smaller than $m_2$, one of the following is true:
    \begin{itemize}
        \item the maximal monomial in $X_1\oplus X_2\oplus X_3$ is equal to the maximal monomial in $X$, but it appears a strictly smaller number of times,
        \item the maximal monomial in $X_1\oplus X_2\oplus X_3$ is smaller than the maximal monomial in $X$.
    \end{itemize}
    Since at each step the number of monomials equal to $m_2$ strictly decreases, repeating this procedure a finite number of times, we find an $\tilde X$ in $A_{d,f}$ that has a smaller maximal monomial than $X$. This contradicts the minimality of $X$. Therefore, we conclude that $\delta\leq d+1$ and so $f\in\Vdd$.
\end{proof}
At this point the proof of Theorem \ref{theorem:boundsolvdeg} is straightforward.
\begin{proof}[Proof of theorem \ref{theorem:boundsolvdeg}]
If $\dreg(\Ff)=+\infty$, the statement is trivially true. Otherwise, by Remark~\ref{remark:grobnerdreg}, we know that $\mathrm{Gbd}_{\sigma}(\Ff)\leq \dreg(\Ff)$. Then, by Theorem \ref{theorem:finVdd}, we obtain that every element that belongs to a reduced Gr\"obner basis is in $\Vdd$. 
\end{proof}
The next example shows that the bound in Theorem \ref{theorem:boundsolvdeg} is optimal.
\begin{example}
Let $1<k\in\N$ and let $\Ff_k=\{x^k+y,y^k+x,xy\}$ a family of polynomial in $\mathbb{K}[x_1,\dots,x_n]$. We have that $\Ff_k^{\mathrm{top}}=\{x^k,y^k,xy\}$ and therefore $\dreg(\Ff_k)=k$. We notice that 
\begin{equation*}
    x=-y^{k-1}(x^k+y)+(y^k+x)+x^{k-1}y^{k-2}(xy),
\end{equation*}
and so $x\in(\Ff_k)$. By symmetry, we obtain that also $y\in(\Ff_k)$. In particular, we have that $\{x,y\}$ is a reduced Gr\"obner basis for $\Ff_k$ for every term order. On the one hand, we observe that $x,y\notin V_{\Ff_k,k}$. On the other hand, we have that
\begin{equation*}
    y^2=y(x^k+y)-x^{k-1}(xy)\in V_{\Ff_k,k+1}\implies x=(y^k+x)-y^{k-2}(y^2)\in V_{\Ff_k,k+1}\implies y\in V_{\Ff_k,k+1}.
\end{equation*}
So, $\{x,y\}\subseteq V_{\Ff_k,k+1}$. This implies that $\sd(\Ff_k)=k+1=\dreg(\Ff_k)+1$, independently of the choice of the degree-compatible term order $\sigma$.
\end{example}
The next proposition generalizes Theorem \ref{theorem:boundsolvdeg} to an arbitrary family of polynomials.
\begin{proposition}\label{corollary:theoremsd}
Let $\mathcal{F}=\{f_1,\dots,f_k\}$ be a family of polynomials and let $\sigma$ be a degree-compatible term order. Then, $$\sd\leq \max\{\dreg(\Ff)+1,\deg(f_1),\dots,\deg(f_k)\}.$$ 
\end{proposition}
\begin{proof}
    Let $\Ff=\{f_1,\dots,f_k\}$ be a family of polynomial, $t=\max\deg(f_i)$ and suppose that $\dreg(\Ff)<t$. We denote by $\mathrm{LT}(f)$ the leading term of $f$. If there exist indices $i,j$ and a monomial $m$ such that $\mathrm{LT}(f_i)=m\mathrm{LT}(f_j)$, we substitute $f_i$ with $f_i-mf_j$. By proceeding in this way, in a finite number of steps we find a family $\Ff'=\{f_1',\dots,f_s'\}\subseteq V_{\Ff,t}$ such that $(\Ff')=(\Ff)$ and for every pair of indices $i,j$ we have that $\mathrm{LT}(f_i)\nmid\mathrm{LT}(f_j)$. This implies that $\max \deg(f_i')\leq \dreg(\Ff')=d'\leq \dreg(\Ff)$. By Theorem \ref{theorem:boundsolvdeg}, we have that $V_{\Ff',d'+1}$ contains a Gr\"obner basis of $\Ff'$. Since $(\Ff)=(\Ff')$, $d'<t$, and $\Ff'\subseteq V_{\Ff,t}$, we conclude that $V_{\Ff,t}$ contains a Gr\"obner basis of $\Ff$.
\end{proof}
As direct consequence of the previous proposition, we obtain the following upper bound on the last fall degree. 
\begin{corollary}\label{corollary:lfddreg}
Let $\mathcal{F}=\{f_1,\dots,f_n\}$ be a family of polynomials. Then, $$\lfd(\Ff)\leq \max\{\dreg(\Ff)+1,\deg(f_1),\dots,\deg(f_k)\}.$$   
\end{corollary}
\begin{proof}
    By Theorem \ref{theore:lfdgorla} and Proposition \ref{corollary:theoremsd}, we obtain that
    \begin{equation*}
        \lfd(\Ff)\leq\max\{\lfd(\Ff),\mathrm{Gbd}_{\sigma}(\Ff)\}=\sd(\Ff)\leq \max\{\dreg(\Ff)+1,\deg(f_1),\dots,\deg(f_k)\}.\qedhere
    \end{equation*}
\end{proof}
In the next corollary, we provide an upper bound on the solving degree that is linear in both the number of variables and the degree of the elements of the system $\Ff$.
\begin{corollary}[Macaulay bound]
    Let $\Ff=\{f_1,\dots,f_k\}\subseteq\mathbb{K}[x_1,\dots,x_n]$ be a family of polynomial equations of degree at least $1$ such that $\dreg(\Ff)<\infty$ (and so $k\geq n$). Let $d_i=\deg(f_i)$ and suppose that $d_1\geq d_2\geq\dots\geq d_k$. Then, for every degree-compatible term order $\sigma$
    \begin{equation*}
        \sd(\Ff)\leq d_1+\dots+d_n-n+2.
    \end{equation*}
\end{corollary}
\begin{proof}
    Since $\dreg(\Ff)<\infty$, there exists $J=\{j_1,\dots,j_n\}\subseteq[k]$ such that
    $\dreg(\Ff)\leq d_{j_1}+\dots d_{j_n}-n+1$.
    Since we assumed that $d_1\geq d_2\geq\dots\geq d_k$, we have that 
    $\dreg(\Ff)\leq d_{1}+\dots d_{n}-n+1$. Moreover for every $f\in \Ff$ we have that 
    $$\deg(f)\leq d_1\leq d_1+(d_2-1)+\dots+(d_n-1)=d_{1}+\dots d_{n}-n+1.$$
    We conclude applying Proposition \ref{corollary:theoremsd}.
\end{proof}
The Macaulay bound was originally proved by Lazard in \cite[Theorem 3]{Laz} for the maximum degree of an element of a reduced Gr\"obner basis of $\Ff$. Notice that our bound (when it can be applied) is stronger. In fact, it depends only on the degree of $n$ polynomials, while the bound of Lazard depends on the degree of $n+1$ polynomials. A similar but weaker bound was given in \cite[Corollary 3.25]{CG}.

We conclude with an explicit upper bound on the complexity of finding a reduced Gr\"obner basis (with respect to a degree-compatible term order) of a system of polynomial equations in terms of the degree of regularity. For instance, this bound applies when we use algorithms such as  MutantXL, MXL2, and similar variants.  
\begin{proposition}
    Let $\Ff=\{f_1,\dots,f_k\}\subseteq\mathbb{K}[x_1,\dots,x_n]$ be a family such that $\deg(f_i)\leq\dreg(\Ff)=d$. Then, a Gr\"obner basis of $\Ff$ can be find in $\mathcal{O}\left((n+1)^{4(d+1)}\right)$ operations in $\mathbb{K}$.
\end{proposition}
\begin{proof}
    Let $N=\binom{d+1+n}{d+1}$ be the number of monomials of degree at most $d=\dreg(\Ff)$. By Theorem \ref{theorem:boundsolvdeg}, to construct a Gr\"obner basis, it is sufficient to find a basis of $\Vdd$. In order to do that we proceed as follows. With linear operations we can ensure that $\mathrm{LT}(f_i)\neq\mathrm{LT}(f_j)$ for all $f_i,f_j$. This is equivalent to computing the reduced row echelon form of a matrix with $N$ columns and $N$ rows, which can be done in $\mathcal{O}(N^3)$ operations in $\mathbb{K}$. We set $\mathcal{B}=\Ff$ and $\mathcal{M}=\{mf_i:f_i\in\Ff,\text{ m monomial of degree at least 1, and }\deg(mf)\leq d+1\}$. At each step, we pick an element $f$ of $\mathcal{M}$, and with linear operations involving only multiple of elements of $\mathcal{B}$ we reduce $f$ until we obtain an element $\bar f$ such that either $\mathrm{LT}(\bar f)=0$ or $\mathrm{LT}(\bar f)\notin\{\mathrm{LT}(b):b\in\mathcal{B}\}$. We remove $f$ from $\mathcal{M}$. Then, if $\mathrm{LT}(\bar f)\notin\{\mathrm{LT}(b):b\in\mathcal{B}\}$, we add $\bar f$ to $\mathcal{B}$ and we add to $\mathcal{M}$ all $m\bar f$ such that $m$ is a monomial of degree at least $1$, and $\deg(m\bar f)\leq d+1$. This operation costs $\mathcal{O}(N^2)$ operations. We repeat this operation until either $\lvert\mathcal{B}\rvert=N$ or $\lvert\mathcal{M}\rvert=0$. Every time that we add a new element to $\mathcal{B}$, we add at most $N$ elements to $\mathcal{M}$. Since in $\mathcal{B}$ we can add at most $N$ elements, this algorithm terminates in at most $N^2$ steps. The cost of each step  is $\mathcal{O}(N^2)$, therefore we have that the a Gr\"obner basis of $\Ff$ can be find in $\mathcal{O}(N^4)$ operations. Since $N\leq (n+1)^{d+1}$, we conclude.
\end{proof}
\section*{Acknowledgement}
The author is grateful to Elisa Gorla for her valuable suggestions.
\bibliographystyle{plain}	
\bibliography{main}
\end{document}